\documentclass[10pt]{article}
\usepackage{graphicx}
\usepackage{amsmath,amsthm}
\usepackage{amssymb}
\usepackage[latin1]{inputenc}
\usepackage{enumerate}
\usepackage{slashed}
\usepackage{pb-diagram}
\usepackage{slashed}
\usepackage{color}
\theoremstyle{plain}
\newtheorem{theo}{Theorem}

\newtheorem{define}[theo]{Definition}
\newtheorem{lemma}[theo]{Lemma}

\newcommand{\tr}{\mathrm{tr}}

\newcommand{\vol}{\mathrm{vol}}
\newcommand{\Tr}{\mathrm{Tr}}
\newcommand{\TN}{\mathrm{TN}}
\newcommand{\cE}{\mathcal{E}}

\newcommand{\calS}{\mathcal{S}}

\newcommand{\cD}{\mathcal{D}}

\newcommand{\cM}{\mathcal{M}}

\newcommand{\frA}{\mathfrak{A}}

\newcommand{\TaubN}{\mathrm{TN}}

\newcommand{\cA}{\mathcal{A}}

\newcommand{\bbR}{\mathbb{R}}

\newcommand{\IND}{\mathrm{Ind}}

\newcommand{\ev}{\mathrm{ev}}

\newcommand{\p}{\partial}

\title{\textbf{ An Application of the Index Theorem for Manifolds with Fibered Boundaries}}
\author{Andr\'es Larra\'in-Hubach\footnote{Part of this work was completed while the author was supported by a Research Council Seed Grant from the University of Dayton.}\\University of Dayton\footnote{alarrainhubach1@udayton.edu}}

\begin{document}
\maketitle
\begin{abstract}
We show how the index formula for manifolds with fibered boundaries of \cite{LMP} can be used to compute the index of the Dirac operator on Taub-NUT space twisted by an anti-self-dual generic instanton connection. \end{abstract}
\section{Introduction}
This note derives the $L^2$-index theorem, first proved in \cite{ChLaSt}, for a Dirac operator twisted by an Anti-Self-Dual (ASD) generic instanton on Taub-NUT space from the index formula on manifolds with fibered boundaries \cite{LMP}. The original motivation for this index problem was to establish the completeness of the bow construction of  Sergey Cherkis \cite{C1}. This construction conjectured an isometry between the moduli space of bow representations and the moduli space of generic ASD-instantons on Taub-NUT space. In  \cite{ChLaSt},  \cite{ChLaSt2} and  \cite{ChLaSt3} the isometry between these moduli spaces is proved. 

An important technical step in verifying this conjecture  is the index theorem mentioned above. The original proof used an adaptation of the powerful machinery developed by Mark Stern \cite{St1} to obtain index formulas on open spaces. The main advantages of this technique is that the argument is self-contained and avoids the use of spectral theory by computing asymptotic information instead. It is  important to mention that the original proof of this index theorem applies not only to Taub-NUT space but also to the family  of multi-centered Taub-NUT spaces. 

Here we give another argument that uses the index theorems for spaces with fibered boundaries developed by several authors \cite{LMP}. These theorems give expressions for the index as a sum of two terms usually called bulk and asymptotic contribution. The bulk is the usual Atiyah-Singer integrand \cite{BGV}, while the asymptotic contribution is generally given in terms of $\eta$-invariants of  Dirac-type operators restricted to the boundary. In our case,  the index theorem for exact-d-metrics of \cite{LMP} applies  and the boundary contribution, given in terms of the Bismut-Cheeger $\hat{\eta}$-form \cite{BC}, can be computed explicitly. The main reason for the success of this computation is that the boundary fibration in our case is just the Hopf fibration $S^3\to S^2.$ Here we restrict to  the single-centered Taub-NUT space as an illustration of the use of the index formula for fibered boundaries.

\subsection*{Acknowledgements}
 I would like to thank Sergey Cherkis and Mark Stern for helpful conversations and several constructive suggestions.

\section{Index Theory on Spaces with Fibered Boundaries}
We start reviewing the statement of the index theorem in \cite{LMP}.

\subsection{The Bismut-Cheeger $\hat{\eta}$-form. }\label{BCsec}
Let $\pi:M\to B$ be a locally trivial fibration  of spin manifolds with base an even dimensional manifold $B$, and   fibers isomorphic to a closed odd-dimensional manifold $Z$. We assume there is a connection on the fibration that induces a   splitting $TM=T_HM\oplus TM/B$ into horizontal and vertical tangent vectors, such that $\pi^*TB$ can be identified with $T_HM$.   Let $g^M=\pi^*g^B\oplus g^{M/B}$ be a Riemannian submersion metric, where $g^B$ is a metric on $TB$  pulled back to $T_HM$, and $g^{M/B}$ denotes a metric on the vertical fibers. 

Let $E\to M$ be a complex vector bundle with unitary connection $\nabla^E$ and curvature $F^E$. The bundle $E$ induces an infinite rank bundle $\pi_*E\to B$ with fibers given by $\Gamma(M_x,E_x)$, where $M_x,\,E_x$ denote the fibers over $x\in B.$ The connection $\nabla^E$ induces a connection on $\pi_*E$ denoted by 
$\nabla^{\pi_*E}$ (see  chapter 10 of \cite{BGV} ). 

We denote by $(\calS^{M/B},\nabla^{M/B})$ the vertical  spinor bundle together with its induced spin connection coming from the metric $g^{M/B}$. We use   $\nabla^{\calS^{M/B}\otimes E}=\nabla^{M/B}\otimes 1+1\otimes \nabla^E$  and the Clifford module structure on  $\calS^{M/B}\otimes E$ with respect to the Clifford algebra of $TM/B$ \footnote{We denote by $c=c^{M/B}$ the Clifford product by elements of $T^*M/B$.},  to construct a family of vertical Dirac operators denoted by $D^{M/B}=c^{M/B}\circ \nabla^{\calS^{M/B}\otimes E}.$




\begin{define}
Let $u$ be a positive  parameter, the   Bismut superconnection, acting on $\Gamma(M,\calS^{M/B}\otimes E)=\Gamma(B,\pi_*(\calS^{M/B}\otimes E))$, is defined by
 \begin{equation}\label{bismut1}
\frA_u=\nabla^{\pi_*(\calS^{M/B}\otimes E)}+\sqrt{u}D^{M/B}-\frac{c(T)}{4\sqrt{u}},
\end{equation}
where $T$ is the torsion form of the fibration $M\to B$ (see proposition 10.15 of \cite{BGV} ). 
\end{define}

\begin{define}
The Bismut-Cheeger Eta form of the vertical family of Dirac operators $D^{M/B}$ is defined by
\begin{equation}\label{Bismut2}
\hat{\eta}(D^{M/B})=\frac{1}{\sqrt{\pi}}\int_0^\infty \Tr^{\mathrm{ev}}\big((D^{M/B}+\frac{c(T)}{4{u}})e^{-\frA_u^2}\big)\frac{du}{2\sqrt{u}}.
\end{equation}
Here $\Tr^{\ev}$ denotes the operator trace on the even form part of $\hat{\eta}$.
\end{define}

\subsection{The Index Formula}\label{indsec}

Consider a Riemannian manifold $(\cM,g^\cM)$ such that its boundary $M=\p\cM$ is the total space of a fibration $\pi:M\to B$ like the one considered in section (\ref{BCsec}). 

We assume that on a tubular neighborhood of the boundary $(a,\infty)_y\times M$, the metric takes the form
\begin{equation}\label{fiberedmetric}
g^\cM=dy^2+\pi^*g^B+e^{-2y}g^{M/B}.
\end{equation}

In the terminology of \cite{LMP}, this is an exact d-metric with boundary defining function $x=e^{-y}.$ 

Let 
\begin{equation}
\cD=\begin{pmatrix}0&\cD^-\\\cD^+&0\end{pmatrix}
\end{equation} 
be a Dirac type operator on  $\calS\otimes \cE\to\cM$ (here $\calS=\calS^+\oplus\calS^-$ denotes the spin bundle), such that its boundary family $D^{M/B}$, acting on $\calS^{M/B}\otimes (\cE|_{M})=\calS^{M/B}\otimes E$, satisfies the technical assumption
\begin{equation}\label{gap0}
\mathrm{Spec (D_b^{M/B}})\cap (-\delta,\delta)=\emptyset,
\end{equation}
for some $\delta>0$ and for every $b\in B.$

With all these definitions in place, we can state the index theorem from  \cite{LMP}.
\begin{theo}
If $D^{M/B}$ satisfies assumption (\ref{gap0}) then 
\begin{equation}\label{ind1}
\IND_{L^2}(\cD^+)=\int_\cM \hat{A}(\cM,g^{\cM})\wedge \mathrm{Ch}(\cE)-\frac{1}{2\pi i}\int_{\p\cM}\hat{A}(B,g^B)\wedge \hat{\eta}(D^{M/B}),
\end{equation}
where $\hat{A}$ denotes the A-hat genus of the corresponding space with metric.\footnote{The $2\pi i$ factor does not appear in the original formula due to the different normalizations used. See \cite{Dai}.} The $\hat{\eta}$-form is computed with respect to the submersion metric $\pi^*g^B+g^{M/B}$
\end{theo}

\subsection{Circle Fibrations}

We only need to compute $\hat{\eta}$ in the case where  $M\to B$ is a $S^1$-principal bundle with a Riemannian submersion metric. Several formulas simplify in this context.

Let $\{f_\alpha\}$ (resp.  $\{f^\alpha\}$) denote an orthonormal frame (coframe) on $B$ and $\{e\},\{e^*\}$ similarly defined on the vertical fibers $M/B$. We use $\{\tilde{f}_\alpha\}$ to denote the horizontal lifts to $T_HM$. Set $c(e^*)=-i$ and denote $\calS^{M/B}$ by $\calS^{S^1}.$

 The torsion form on a circle fibration can be computed explicitly (see \cite{Sav}). We get
\begin{equation}\label{torsion}
T(\tilde{f}_\alpha,\tilde{f}_\beta)=de^*(\tilde{f}_\alpha,\tilde{f}_\beta)=R(f_\alpha,f_\beta),
\end{equation}
where $R$ is the curvature of the $S^1$-connection inducing the splitting $TM=T_HM\oplus TB$.

Let $z$ denote a Grassmann variable i.e. $z^2=0$. A well-known trick by Bismut and Cheeger  rewrites $\frA_u^2$ as follows
\begin{equation}\label{trick}
\begin{split}
-u\big(\nabla^{\pi_*(\calS^{S^1}\otimes E)}_e&+\frac{R}{4u}-\frac{iz}{2\sqrt{u}}\big)^2+\sqrt{u}F^{E}(f_\alpha,e)f^\alpha\wedge e^*+\frac{1}{2}F^{E}(f_\alpha,f_\beta)f^\alpha\wedge f^{\beta}\\
&=\frA_u^2-z(\sqrt{u}\cD^{M/B}+\frac{c(T)}{4\sqrt{u}})
\end{split}
\end{equation}
This is just  \cite[4.68-4.70]{BC}  adapted to circle fibrations. An important simplification of the original formula is that the scalar curvature of the circle  fibers vanishes.

\section{Taub-NUT Space}
In this section we state the necessary definitions and results from \cite{ChLaSt}.
\subsection{Definition and Basic Properties}
We start by defining the single-centered Taub-NUT space, denoted by TN henceforth. It is a hyperK\"ahler 4-manifold that, outside a compact set, is  a circle fibration over $\mathbb{R}^3$. It has coordinates $\{x_1,x_2,x_3,\tau\}$, where the $x_j$ parameterize $\mathbb{R}^3$ and $\tau=\tau+2\pi$ parameterizes the circle fiber. Define the TN metric, denoted by $g_{\TN}$, by \footnote{Usually, $V=l+\frac{1}{2r}$ for some fixed $l>0$. We simplify the computations by setting $l=1$.}
\begin{equation}\label{original}
g_{\TN}=V(dx_1^2+dx_2^2+dx_3^2)+\frac{1}{V}(d\tau+\omega)^2,
\end{equation}
where $V=1+\frac{1}{2r}$ with $r=\sqrt{x_1^2+x_2^2+x_3^2}$ and $\omega$ is a one-form such that $\star_3 dV=d\omega.$ \footnote{$\star_3$ denotes the Hodge start in $\bbR^3$.}  \footnote{We use the orientation $dx_1\wedge dx_2\wedge dx_3\wedge d\tau$. }


Using polar coordinates about the origin, we can rewrite the $\mathbb{R}^3$-metric as $dr^2+r^2g_{S^2}$, where $g_{S^2}$ is the usual round metric on the two-sphere.

\subsection{Generic ASD Instantons on TN}

Let $\mathcal{E}\to \TaubN$ be a unitary bundle of rank $m$. A generic Anti-Self-Dual (ASD) instanton on it, is a unitary connection $A$ on $\mathcal{E}$ such that its $L^2$-curvature form satisfies $F_A=-\star_{\TN} F_A$. Here $\star_\TN$ is the Hodge star of the original metric $g_\TN$. The genericity is a technical condition explained in \cite{ChLaSt}. In  (\cite{ChLaSt} theorem B) it is proved that there is a frame of $\mathcal{E}$ such that an  ASD generic instanton on TN has the following asymptotic form outside a compact set
\begin{equation}\label{inst}
\cA=-i\,\mathrm{Diag}\big((\lambda_j+\frac{m_j}{2r})\frac{d\tau+\omega}{V}+\eta_j\big)+\mathcal{O}(r^{-2}),
\end{equation}
where $\eta_j$ is a connection one-form on a complex line bundle $W(j)$ over $S^2$. The $\lambda_j$ are related to the asymptotic eigenvalues of the holonomy of $A$ along $\tau$-circles and they are pairwise distinct and constant.  Notice that, near infinity, the bundle $\mathcal{E}$ decomposes as a direct sum of eigenline-bundles of the holonomy of $\cA$.

We assume a stronger genericity assumption by imposing 
\begin{equation}\label{a}
e^{2\pi i \lambda_j}\neq 1,
\end{equation}
for every $j$. This is required to guarantee  the Fredholmness of the Dirac operator twisted by $\cA$. \cite{ChLaSt}

\subsection{Modifications of the Metric}

In order to use the index formula of {\cite{LMP}}, we first need to modify the metric on $\TaubN$, preserving the index, to resemble a fibered cusp metric at infinity of the type defined in section \ref{indsec}.

First, we apply a conformal transformation to $g_{\TN}$ with a conformal factor that equals $e^{2u}=V^{-1}r^{-2}$ for $r$ large. We  consider $e^{2u}$ to be equal to the identity around the origin. The resulting metric, near infinity and using the change of variables $r=e^{y}$, equals
\begin{equation}\label{Firstmod}
g'=dy^2+g_{S^2}+\frac{1}{V^2e^{2y}}(d\tau+\omega)^2.
\end{equation}
Now we use a smooth homotopy between $g'$ and a metric $g$. The homotopy replaces $V$ by $V_t=1+\frac{t}{2e^y}$ for $y$-large and it equals the identity near the origin. The resulting metric at $t=0$, for $y$ large, has the form
\begin{equation}\label{metric}
g=dy^2+\pi^*g_{S^2}+\frac{1}{e^{2y}}(d\tau+\omega)^2
\end{equation}
and it is an exact d-metric  {\cite{LMP}}. The boundary fibration in this case is the Hopf fibration $\pi: S^3_\infty\to S^2_\infty$.
\begin{lemma}
Let $\mathcal{R}_g$ be the Riemannian curvature of the metric $g$, then
$$\frac{1}{192\pi^2}\int_{TN}\tr\, \mathcal{R}_g\wedge \mathcal{R}_g=\frac{1}{12}$$
\end{lemma}
\begin{proof}
The space TN with metric $g$ is contractible and therefore topologically trivial. This implies that its tangent bundle has a global trivialization. Fixing a trivialization, we can define the Chern-Simons form $CS(g)$ such that $dCS(g)= \frac{1}{192\pi^2}\tr\, \mathcal{R}_g\wedge \mathcal{R}_g.$ 

Let $B$ be a ball around the origin such that the metrics $g_{\TaubN}$ and $g$ coincide in on an open neighborhood of it. Then, (see Lemma 32 of \cite{ChLaSt})
\begin{align*}
\frac{1}{192\pi^2}\int\tr\, \mathcal{R}_g\wedge \mathcal{R}_g&=\frac{1}{192\pi^2}\int_B\tr\, \mathcal{R}_g\wedge \mathcal{R}_g+\frac{1}{192\pi^2}\int_{\TaubN\setminus B}\tr\, \mathcal{R}_g\wedge \mathcal{R}_g\\
&=\frac{1}{192\pi^2}\int_{B}\tr\, \mathcal{R}_{g_\TaubN}\wedge \mathcal{R}_{g_\TaubN}+\frac{1}{192\pi^2}\int_{\TaubN\setminus B}\tr\, \mathcal{R}_g\wedge \mathcal{R}_g\\
&=-\frac{1}{192\pi}\int_{S^2_0}\nabla_n\big(\frac{|\nabla V|^2}{V^3}\big)\vol_{S^2_0}+\int_{S^3_\infty}CS(g),
\end{align*}
where $S^2_0$ is a small two-sphere around the origin with outward normal vector $n$. The form $CS(g)$ decays exponentially and the last boundary integral equals zero. The first summand equals $\frac{1}{12}.$

\end{proof}

\subsection{The Twisted Dirac Operator}

Let $\cD_A$ be the twisted Dirac operator with respect to the metric $g_\TN$. We write $\cD_{g'}$ and $\cD_g$ to denote the corresponding twisted Dirac operators with respect to the other metrics\footnote{To simplify notation, we set $\cD_g=\cD_g^A$ etc...}. Let $\{\cD_t\}_{0\leq t\leq 1}$ be the family of Dirac operators associated to the metrics in the homotopy between $g'$ and $g$.

All of these operators admit a decomposition according to chirality as
\begin{equation}\label{chiral}
\cD=\begin{pmatrix}0&\cD^-\\\cD^+&0\end{pmatrix},
\end{equation}
where $\cD^\pm: \Gamma(\calS^\pm\otimes \cE)\to \Gamma(\calS^\mp\otimes \cE)$. Here we denote by $\calS=\calS^+\oplus \calS^-$ the spinor bundle of TN.

The $L^2$-indices of these operators turn out to be the same
\begin{lemma}
\begin{equation}\label{indfund}
\IND_{L^2}\cD^+_A=\IND_{L^2}\cD^+_{g'}=\IND_{L^2}\cD^+_{g}
\end{equation}
\end{lemma}
\begin{proof}
The first equality is proved in \cite{ChLaSt}. For the second one, since the $\lambda_j$ are constant and satisfy  assumption (\ref{a}), the  $\cD_t$ satisfy the hypotheses of lemma 22 in \cite{ChLaSt} for every $0\leq t\leq 1$. This implies that $\{\cD_t\}$ is a homotopy within the space of Fredholm operators so the index is preserved.


\end{proof}
From now on, we only use the operator $\cD_g$ and its chirality components $\cD_g^{\pm}$.

The corresponding vertical family of Dirac operators is
\begin{equation}\label{vertical}
D^{M/B}=D_g^\p=\oplus_{j=1}^mD^{\p,{\lambda_j}}=\oplus_{j=1}^m(-i)(\p_\tau-i{\lambda_j}).
\end{equation}
Remember that here we use the submersion metric $g_{S^2_\infty}+(d\tau+\omega)^2$ on $S^3_\infty.$
\begin{lemma}\label{gap}
The family $D^\p_g$ parametrized by points in $S^2_\infty$ satisfies the spectral gap asumption in \cite{LMP}. That is, Spec$\{D^\p_{g,p}\}_{p\in S^2}\cap (-\delta,\delta)=\emptyset$ for some $\delta>0$.
\end{lemma}
\begin{proof}
Given $p\in S^2_\infty$, the boundary operator $D^\p_p$ is a multiple of a  direct sum of operators of the form  $-i(\p_\tau-i\lambda_j)$. Since  $e^{2\pi i\lambda_j}\neq 1$,  the operator is invertible. The $\lambda_j$ are constant so we can take $\delta=\frac{1}{2}\mathrm{min}_j|\lambda_j|$.
\end{proof}
\section{The Index of $D_g$}

From the previous sections we see that $\cD_g$ satisfies the assumptions of the index formula in \cite{LMP}. In our case, we obtain
\begin{equation}\label{indthm}
\IND_{L^2}D^+_g=\int_{TN}\big(\frac{\mathrm{Rank}(\mathcal{E})}{192\pi^2}\tr\, \mathcal{R}_g\wedge \mathcal{R}_g-\frac{1}{8\pi^2}\tr \,F_A\wedge F_A\big)-\frac{1}{2\pi i}\int_{S^2_\infty}\hat{\eta},
\end{equation}
where $\hat{\eta}\in \Omega^\ev(S^2_\infty)$ is an even form called the Bismut-Cheeger $\hat{\eta}$-form, computed with respect to the metric $g_{S^2}+(d\tau+\omega)^2$ on the boundary fibration $S^3_\infty\to S^2_\infty$.

It remains to compute $\hat{\eta}$ in terms of the asymptotic form of the twisting connection $\cA|_{S^3_\infty}=A$.

\subsection{Explicit Computation}
Again,  the boundary fibration on TN  equals the  Hopf fibration $S^3\to S^2$ with metric $g_{S^2}\oplus (d\tau+\omega)^2$. Notice that $d\tau+\omega$ is a connection one-form for the fibration so the curvature equals
\begin{equation}\label{s1curv}
R=d\omega=-\frac{1}{2}\vol_{S^2_\infty},
\end{equation}
where $\vol_{S^2_\infty}$ is the volume form on the two-sphere.

 The corresponding bundle $\cE|_{S^3_\infty}=E\to S^3_\infty$ is  the restriction of the instanton bundle $\cE$ to the boundary fibration inherits a connection of the form
\begin{equation}\label{inst2}
A=-i\,\mathrm{Diag}\big(\lambda_j({d\tau+\omega})+\pi^*(\eta_j)\big).
\end{equation}
This splitting of $A$ and the results of \cite{ChLaSt} imply that $E\to S^3_\infty$ can be decomposed as a direct sum 
\begin{equation}\label{direct}
E=\oplus_{j=1}^m \pi^*W(j),
\end{equation}
where the  $W(j)$ are line bundles over the base $S^2_\infty$ with connection form $\eta_j$. Note that we identify $e^*=(d\tau+\omega)$. 


The Bismut superconnection in this case equals 
\begin{equation}\label{bismut3}
\frA_u=\oplus_{j=1}^{m}\big(\nabla^{\pi_*(\calS^{S^1}\otimes W(j) )}+\sqrt{u}D^{\p,{\lambda_j}}\big)-\frac{c(T)}{4\sqrt{u}}
\end{equation}
The identity (\ref{trick}) simplifies further to
\begin{equation}\label{trick2}
-u\big(\nabla_e+\frac{R}{4u}-\frac{iz}{2\sqrt{u}}\big)^2+F^{E}(f_1,f_2)f^1\wedge f^{2}=\frA_u^2-z\big(\sqrt{u}\cD^{M/B}+\frac{c(T)}{4\sqrt{u}}\big),
\end{equation}
where $\{f_1,f_2\}$ is an oriented orthonormal frame on $S^2_\infty=S^2$ and $F^E$ is the curvature of $A$. Notice that $F^E(e,f_\alpha)=0$ for  $\alpha=1,2.$

The purpose of (\ref{trick}) and (\ref{trick2}) is the following lemma.
\begin{lemma}
Let $\Tr^z(a+zb)=\Tr\, b$ then
\begin{equation}\label{trick3}
\Tr^\ev\big(D^\p+\frac{c(T)}{4u}\big)e^{-\frA_u^2}=u^{-1/2}\Tr^z\exp \big({u(\nabla_e+\frac{R}{4u}-\frac{iz}{2\sqrt{u}})^2-F^E}\big),
\end{equation}
where $F^E=F^\mathcal{E}(f_1,f_2)f^1\wedge f^{2}$.
\end{lemma}
\begin{proof}
Exponentiate both sides of (\ref{trick2}) and notice that $\Tr^z(e^{a+bz})=\Tr \,e^ab=\Tr \,b e^a.$
\end{proof}
Replacing this in the definition of  $\hat{\eta}$ gives
\begin{equation}\label{replace}
\hat{\eta}(D_g^\p)=
\frac{1}{\sqrt{\pi}}\int_0^\infty \Tr^z\exp\big(u(\nabla_e+\frac{R}{4u}-\frac{iz}{2\sqrt{u}})^2\big)\frac{du}{2u}\,e^{-F^E}
\end{equation}
Let $\Lambda=\mathrm{Diag}(\lambda_j)$. In order to compute $\Tr^z$, we use a Fourier mode decomposition of  in the circle variable $\tau$ of sections of $E$. The action of $\nabla_e=\p_\tau-i\Lambda$ on the $k^{\mathrm{th}}$-mode equals $ik-i\Lambda$. Therefore,
\begin{align*}
\hat{\eta}(D_g^\p)&=
\frac{1}{\sqrt{\pi}}\int_0^\infty \sum_{k\in\mathbb{Z}}\Tr^z\big(\exp(-u(k-\Lambda+\frac{R}{4ui}-\frac{z}{2\sqrt{u}})^2)\big)\frac{du}{2u}\,e^{-F^E}\\
&=\frac{1}{\sqrt{\pi}}\int_0^\infty \tr_E \sum_{k\in\mathbb{Z}}(k-\Lambda+\frac{R}{4ui})e^{-u(k-\Lambda+\frac{R}{4ui})^2}\frac{du}{2\sqrt{u}}\,e^{-F^E}\\
\end{align*}
We  need the Poisson summation formula \cite{ChLaSt}.
$$\sum_{k\in\mathbb{Z}}(k+a)e^{-4\pi^2s(k+a)^2}=\sum_{p\geq1}2p\sin(2\pi p a)(4\pi s)^{-3/2}e^{-\frac{p^2}{4s}}.$$
to obtain
\begin{equation}
\hat{\eta}(D_g^\p)=\pi\int_0^\infty\tr_E\sum_{p\geq 1}p\sin\big(2\pi p(-\Lambda+\frac{R}{4ui})\big)e^{-\frac{\pi^2 p^2}{u}}\frac{du}{u^2}\, e^{-F^E}
\end{equation}
Solving the $u-$integral and simplifying gives
\begin{equation}
\hat{\eta}(D_g^\p)=\tr_E\big(\sum_{p\geq 1}\frac{-\sin 2\pi p\Lambda}{\pi p}+\frac{R}{2i}\sum_{p\geq 1}\frac{\cos 2\pi p\Lambda}{\pi^2 p^2}\big)\, e^{-F^E}
\end{equation}
Recall now  the Fourier series expansions of Bernoulli polynomials \begin{equation}\label{almost}
\hat{\eta}(D_g^\p)=\tr_E\big((\{\Lambda\}-\frac{1}{2})+(\{\Lambda\}^2-\{\Lambda\}+\frac{1}{6})\frac{R}{2i}\big)\, e^{-F^E}
\end{equation}
Finally, integrating over $S^2_\infty$ and taking $\tr_E$ we get 
\begin{equation}\label{punchline}
\frac{1}{2\pi i}\int_{S^2_\infty}\hat{\eta}=-(\{\Lambda\}-\frac{1}{2})\frac{1}{2\pi i}\int_{S^2_\infty}\tr_E F^E+\frac{m}{2}(\{\Lambda\}^2-\{\Lambda\}+\frac{1}{6})
\end{equation}
Going back to (\ref{indthm}) we deduce the  result of \cite{ChLaSt}.
\begin{theo}
 \cite{ChLaSt} Theorem 41.
\begin{align}\label{indthm2}
\IND_{L^2}\cD_g^+=-\frac{1}{8\pi^2}\int_{TN}\tr \,F_A\wedge F_A+
(\{\Lambda\}-\frac{1}{2})\frac{1}{2\pi i}\int_{S^2_\infty}\tr_EF^E-\frac{m}{2}(\{\Lambda\}^2-\{\Lambda\}),
\end{align}
where $m=\mathrm{Rank}(\mathcal{E})$.
\end{theo}

\end{document}